\documentclass{amsart}

\usepackage[latin1]{inputenc}
\usepackage{color,mathrsfs,amssymb,hyperref,pgf,tikz} 

\newcommand{\C}{{\mathbb C}}

\newcommand{\N}{{\mathbb N}}

\newcommand{\R}{{\mathbb R}}

\newcommand{\Z}{{\mathbb Z}}

\newcommand{\fa}{{\mathfrak a}}

\newcommand{\fd}{{\mathfrak d}}

\newcommand{\fm}{{\mathfrak m}}

\newcommand{\cH}{{\mathcal H}}

\newcommand{\cL}{{\mathcal L}}

\newcommand{\scrB}{\mathscr{B}}

\newcommand{\scrS}{\mathscr{S}}

\newcommand{\ga}{\alpha}
\newcommand{\gb}{\beta}
\renewcommand{\gg}{\gamma}
\newcommand{\gG}{\Gamma}
\newcommand{\gd}{\delta}

\newcommand{\gve}{\varepsilon}
\newcommand{\gk}{\varkappa}
\newcommand{\gl}{\lambda}

\newcommand{\go}{\omega}

\newcommand{\gvp}{\varphi}
\newcommand{\gt}{\theta}

\DeclareMathOperator{\ctg}{ctg}
\DeclareMathOperator{\tg}{tg}

\DeclareMathOperator{\diag}{diag}

\DeclareMathOperator{\im}{Im}

\DeclareMathOperator\Mp{Mp}
\DeclareMathOperator{\op}{op} 

\DeclareMathOperator{\ord}{ord}
\DeclareMathOperator\Sp{Sp}

\DeclareMathOperator{\re}{Re}
\DeclareMathOperator{\res}{res}
\DeclareMathOperator{\sgn}{sgn}

\DeclareMathOperator{\Tr}{Tr} 
\DeclareMathOperator{\tr}{tr}

\DeclareMathOperator{\U}{U}

\newtheorem{theorem}{Theorem}[section]

\newtheorem{lemma}[theorem]{Lemma}
\newtheorem{proposition}[theorem]{Proposition}

\theoremstyle{definition}
\newtheorem{remark}[theorem]{Remark}
\newtheorem{definition}[theorem]{Definition}
\newtheorem{example}[theorem]{Example}

\def\skp#1{\langle#1\rangle}
\def\mat#1{\begin{pmatrix}#1\end{pmatrix}}
\def\Res{\mathrm{Res}_{z=0}}

\renewcommand{\labelenumi}{(\roman{enumi})}

\numberwithin{equation}{section}

\begin{document}

\title[Trace Expansions and Equivariant Traces]{Trace Expansions and Equivariant Traces on  an \\ Algebra of Fourier Integral Operators on $\R^n$}

\author{Anton Savin}
\address{A. Savin. Peoples' Friendship University of Russia (RUDN University),
6 Miklukho-Maklaya St, Moscow, 117198, Russia}
\email{a.yu.savin@gmail.com} 

\author{Elmar Schrohe} 
\address{E. Schrohe. Leibniz University Hannover, Institute of Analysis, Welfengarten 1, 30167 Hannover, Germany }
\email{schrohe@math.uni-hannover.de} 
\date{}

\thanks{ES gratefully acknowledges the support of Deutsche Forschungsgemeinschaft through grant SCHR 319/10-1.  AS is grateful for support to the Russian Foundation for Basic Research, project Nr.~21-51-12006.}


\date{}

\begin{abstract}
We consider the operator algebra $\mathscr A$ on $\mathscr S(\mathbb R^n)$ generated by the Shubin type pseudodifferential operators, the Heisenberg-Weyl operators and the lifts of the unitary operators on $\mathbb C^n$ to metaplectic operators. 

With the help of an auxiliary operator in the Shubin calculus, we find trace expansions for these operators in the spirit of Grubb and Seeley. Moreover, we can define a noncommutative residue generalizing that  for the Shubin pseudodifferential operators and obtain a class of localized equivariant traces on the algebra.
\end{abstract}

\maketitle
\tableofcontents

\section{Introduction}
On the space $\mathscr S(\R^n)$ of Schwartz functions on $\R^n$ we study the algebra 
 $\mathscr A$ of all operators 
\begin{eqnarray}
\label{eq_D}
%\lefteqn{
D= \sum R_{g} T_{w} A,
\end{eqnarray}
where the sum is finite, the operators $A$ are classical Shubin type pseudodifferential operators of integer order, the operators $T_w$, $w\in \C^n$,  are Heisenberg-Weyl operators, acting on $u\in \scrS(\R^n)$ by
$$T_w u(x)  = e^{ikx-iak/2} u(x-a), \quad w=a-ik,$$
and finally the operators $R_g$, $g\in \Sp(n)\cap \mathrm O(2n) \cong \U(n)$ are lifts of unitary operators on $\C^n\cong T^*\R^n$ to elements of the complex metaplectic group on $L^2(\R^n)$. More details will be given, below. 
We have denoted by $\Sp(n)$ and $\mathrm O(2n)$ the groups of symplectic and orthogonal matrices, respectively, on $T^*\R^n\cong\R^n\times \R^n$ and identified $T^*\R^n$ with $\C^n$ via $(x,p)\mapsto p-ix$.  

The group $\Mp(n)$ of metaplectic operators on $L^2(\R^n)$ is generated by the operators $S_q=e^{i\op^w q}$, where $\op^wq$ is the Weyl quantization of a homogeneous quadratic form $q=q(x,p)$ on $\mathbb \R^n\times\R^n$. 
The operator $S_q$ defines a symplectic map on $T^*\R^n$ by evaluating the Hamiltonian flow generated by the function $q$ at time $1$. In fact, this yields a surjection 
$\pi: \Mp(n)\to \Sp(n)$.
It is well-known that this is a double covering. 
For more details on the metaplectic group we refer to Leray \cite{Ler8} and de Gosson \cite{deG1} as well as to \cite{SaSch2}. 

The complex metaplectic group $\Mp^c(n)$  is generated by the operators $S_{q,\phi}=e^{i(\op^w q+\phi)}$  with $q$ as before and $\phi\in \R$. As shown in \cite[Proposition 1]{SaSch2}, the map 
$$R: \U(n)\to \Mp^c(n),  g \mapsto R_g = \pi^{-1}(g) \sqrt{\det g},$$
first defined for $g$ close to the identity $g=I\in \U(n)$ with the choices $\pi^{-1} (I)=I$ and $\sqrt{\det I}=1$, extends continuously to all of $\U(n)$.  These are the operators used above. The metaplectic operators form a special class of Fourier integral operators, compatible with the Shubin pseudodifferential operators in the sense that, for a Shubin symbol $a$ and a metaplectic operator $S$,    
$$S^{-1} \op^w(a) S = \op^w(a\circ \pi(S)). $$

This algebra has been studied before in \cite{SaSch3}, where explicit formulae for the Chern-Connes character in the local index formula of Connes and Moscovici \cite{CoMo1} were derived and applications to noncommutative tori and toric orbifolds were given. 
The subalgebra without the Heisenberg-Weyl operators has been used in \cite{SaSch2} in the context of index theory. 
The main point of considering these algebras lies in the chance of being able to perform concrete computations, in particular in connection with index theory for operator algebras associated with groups of quantized canonical transformations as they were considered in \cite{SaSchSt4}, \cite{SaSch1} by the authors and B. Sternin or in  \cite{GKN} by Gorokhovsky, de Kleijn and Nest. 

In the present article, we are interested in trace expansions. 
We choose an elliptic auxiliary operator $H$ of positive order in the Shubin calculus (a good choice would be the harmonic oscillator $H_0=\frac12(|x|^2-\Delta)$), for which the resolvent $(H-\gl)^{-1}$ exists for large $\gl$ in a sector about the negative real axis. 
Given an operator $D=R_gT_wA$ as above, the operator $D(H-\lambda)^{-K}$ will be of trace class for sufficiently large $K$. Moreover, the trace will be a holomorphic function of $\lambda$. We will show that, as $\gl\to \infty$ in the sector, we have an expansion  
\begin{eqnarray}\label{resolvent_exp}
\Tr(R_gT_wA(H-\lambda)^{-K}) 
&\sim& \sum_{j=0}^\infty c_j (-\lambda)^{(2m+\ord A-j)/\ord H-K} \\
&&+\sum_{j=0}^\infty (c_j' \ln (-\lambda)+c_j'')(-\lambda)^{-j-K}\nonumber
\end{eqnarray}
with suitable coefficients $c_j, c_j', c_j''$, $j=0,1,\ldots$. Here, $\ord A $ and $\ord H$ denote the orders of $A$ and $H$, respectively, and $m$ is the complex dimension of the fixed point set of $g$. 
In this expansion, the coefficients $c_j$ and $c_j'$ are `local' in the sense that they can be determined from the terms in the asymptotic expansion of the symbols of $D$ and $H$, respectively. They can therefore -- at least in principle -- be computed explicitly,  while the $c_j''$ are 'global'; they depend also on the residual part of the symbol  and hence cannot be determined in general. 
This can be seen as a generalization of the corresponding result by Grubb and Seeley \cite[Theorem 2.7]{GS1}.  

The derivation of this expansion is the heart of the analysis. We obtain from it two other results. The first  concerns the operator zeta function $\zeta_{R_gT_wA}(z) = \Tr(R_gT_wAH^{-z})$. Here, we additionally assume 
$(H-\gl)^{-1}$ to exist in the closed sector, including $\gl=0$. 
The zeta function then is defined  and  holomorphic for $\re z>(2n+\ord A)/\ord H$. Moreover, it has a meromorphic  extension to the whole complex plane with at most simple poles. In fact, we have the pole structure 
\begin{eqnarray}\label{trace_exp}
%\label{}
\lefteqn{\Gamma(z)\zeta_{R_gT_wA}(z)}\\
&\sim &\sum_{j=0}^\infty \frac{\tilde c_j}{z-(2m+\ord A-j)/\ord H}
+ \sum_{j=0}^\infty \Big(\frac{\tilde c_j'}{(z+j)^2} + \frac{\tilde c_j''}{z+j}\Big)
\nonumber
\end{eqnarray}
with suitable coefficients $\tilde c_j, \tilde c_j'$ and $\tilde c_j''$ related to those above by universal constants. 
The analysis of such operator zeta functions goes back to Seeley's classical article on complex powers of an elliptic operator \cite{Seeley67}, where $\Tr(H^{-z})$  was studied for an elliptic pseudodifferential operator $H$. 
Grubb and Seeley in \cite{GS2} also showed how this expansion can be derived from the expansion of the resolvent trace under rather mild conditions. %, and we will proceed in this way. 

Assuming, moreover, that the sector in which the resolvent is holomorphic is larger than the left half plane, we can study the `heat trace'  $\Tr\big(R_gT_wAe^{-tH}\big)$. It exists for all $t>0$ and has an expansion 
\begin{eqnarray}
\label{exp_exp}
%\lefteqn{
\mbox{\;\;\;}\Tr( R_gT_wAe^{-tH})  \sim \sum_{j=0}^\infty \tilde c_j t^{(j-2m-\ord A)/\ord H} + 
\sum_{j=0}^\infty (-\tilde c'_j\ln t + c_j'')t^j.
\end{eqnarray}
Note that the coefficients coincide with those in \eqref{trace_exp}.
Just as the result on operator zeta functions, this statement is a consequence of the resolvent expansion together with abstract arguments. 

The coefficients $c_0'$ and $\tilde c_0'$ play a remarkable role in these expansions. 
For one thing, they are the same; moreover, after multiplication by $\ord H$ they are completely independent of the auxiliary operator $H$. 
In fact they allow us to define an analog of the noncommutative residue for our algebra. 
Restricted to the Shubin pseudodifferential operators it coincides with the residue introduced by 
Boggiatto and Nicola \cite{BoNi} for a more general class of anisotropic pseudodifferential operators on $\R^n$. 

We recall that the noncommutative residue for the algebra of classical pseudodifferential operators on a closed manifold, discovered independently by Wodzicki \cite{Wod2}  and Guillemin \cite{Gu} defines the unique trace -- up to multiples -- on this algebra (see e.g.~\cite{FGLS} for a simple proof). The same is true for the residue of Boggiatto and Nicola. 

Our algebra $\mathscr A$ is the algebraic twisted crossed product of the algebra $\Psi$ of all Shubin type pseudodifferential operators with the group $\C^n\rtimes \U(n)$ via the representations of $\C^n$ by the operators $T_w$ and of $\U(n)$ by the operators $R_g$. Given a discrete subgroup
$G$ of $\C^n\rtimes \U(n)$ and an element $(w_0,g_0)$ of $G$, we define the residue trace 
of the operator $D=\sum R_gT_wA$  localized at the conjugacy class $\skp{(w_0,g_0)}$ by 
\begin{eqnarray}
\label{eq_res}
%\lefteqn{
\res_{\skp{(w_0,g_0)}} D = \ord H\sum_{(w,g)\in \skp{(w_0,g_0)}} c_0'(R_gT_wA)
\end{eqnarray}
and show that these are actually traces on the algebra $\mathscr A$. 

The present article in this sense  continues the work of Dave \cite{D1}. He defined  equivariant noncommutative residues for the algebra $\Psi_{\rm cl}(M)\rtimes \gG$ of operators on a closed manifold $M$, generated by the classical pseudodifferential operators and a finite group $\gG$ of diffeomorphisms of $M$ and computed the cyclic homology in terms of the de Rham cohomology of the fixed point manifolds $S^*M^g$. 
In the same vein, Perrot \cite{Perrot14} studied the local index theory for shift operators associated to non-proper and non-isometric actions of Lie groupoids; he also computed the residue of the corresponding operator zeta functions in zero.   

{\bf Organisation of the article.} 
We recall in Section \ref{sect_algebra} the essential facts about the operator classes. Section \ref{main} contains the statements of the main results: Theorem \ref{thm_resolvent} on the resolvent expansion, Theorem \ref{thm_zeta} on the structure of the singularities of the operator zeta function, Theorem \ref{thm_exp} on the `heat trace' expansion and   Theorem \ref{residue} on the form of the noncommutative residue. Theorems \ref{thm_zeta} and \ref{thm_exp}  follow from 
Theorem \ref{thm_resolvent} by  abstract results in \cite{GS2}.  
Section \ref{sect_resolvent}  contains a sketch of the proof of the resolvent expansion, while Section \ref{sect_residue} treats the residue; full details will be given elsewhere. 
The appendix contains the essential material on the weakly parametric calculus of  Grubb and Seeley.

\section{The Operators}\label{sect_algebra}

\subsection*{Shubin type pseudodifferential operators}
A smooth function $a$ on $\R^n\times \R^n$ is called a (Shubin type)  pseudodifferential 
symbol of order $\fa\in \R$, provided that, for all multi-indices $\ga$, $\gb$, it 
satisfies the estimates 
$$D^\ga_pD^\gb_x a( x,p) \lesssim (1+|x|^2+|p|^2)^{(\fa-|\ga|-|\gb|)/2}, \quad(x,p)\in \R^n\times \R^n.$$
From such a symbol, we can define the associated pseudodifferential operators $\op a$ and 
$\op^w a$, respectively,  by standard and Weyl quantization.

By $\Psi^\fa$ we denote the space of all Shubin type pseudodifferential operators of order $\fa$ and by $\Psi$ the union over all $\fa\in \R$.
An important example is the operator
\begin{eqnarray}\label{def_ham}%
H_0 = \frac12(|x|^2-\Delta) \in \Psi^2,  
\end{eqnarray}
the so-called harmonic oscillator, which is selfadjoint and positive on $L^2(\R^n)$. 

These operators act on the Sobolev spaces $\mathcal H^s(\R^n)$, $s\in \R$, consisting of all tempered distributions $u$ such that $\op((1+|x|^2+|p|^2)^{s/2})u\in L^2(\R^n)$:  In fact, $\Psi^\fa\hookrightarrow \scrB(\cH^s(\R^n),\cH^{s-\fa}(\R^n)) $ for all  $s,\fa\in \R$. In particular, zero order  operators 
are bounded on $L^2(\R^n)$ and those of  order $<-2n$ are of trace class.

We will work here mostly with classical symbols, i.e.~symbols $a$ of order $\fa\in \Z$ that have an asymptotic expansion 
$$a \sim \sum_{j=0}^\infty a_{\fa-j}  $$
with $a_{\fa-j}$ smooth and positively homogeneous in $(x,p)$ of degree $\fa-j$ for 
$|(x,p)|\ge 1$. We call $a_\fa$ the principal symbol and write $\Psi^\fa_{\rm cl} $ 
and $\Psi_{\rm cl}$ for the classes of operators with classical symbols, respectively. 

A Egorov type theorem holds: 
Given an element $S$ in the metaplectic group $\Mp(n)$ and 
$A={\rm op}^w a$ a Weyl-quantized classical Shubin  pseudodifferential operator  with symbol $a$, then $S^{-1}AS$
is the Weyl-quantized Shubin  pseudodifferential operator with symbol $a\circ \pi(S)$, where $\pi(S)\in \Sp(n)$ is the canonical transformation associated with $S$, see 
\cite[Theorem 7.13]{deG1}. As the principal symbol of  ${\rm op}^w(a)$ coincides with that of ${\rm op}(a)$, we find the following relation for principal symbols: 
\begin{eqnarray}\label{Egorov}%
\sigma (S^{-1}AS) = \sigma(A)\circ \pi(S).
\end{eqnarray}

%$A- SAS^{-1} = S(S^{.-1}A -AS^{-1} )$

The above operators need not be scalar, it is often important to also consider systems of operators  over $\R^n$, see e.g. \cite{SaSch3}.

\subsection*{Heisenberg-Weyl operators.} For  $w=a-ik\in\mathbb{C}^n$,  $a,k\in \mathbb{R}^n$, we define %the operators
\begin{eqnarray}%\label{def_Tz}%
T_w u(x)=e^{ikx-iak/2}u(x-a).
\end{eqnarray}
We note that for $v,w\in \C^n$
\begin{equation}
\label{eq-comm1}
T_{v}T_{w}=e^{-i\im( v,w)/2}T_{v+w},\quad \text{ where } (v,w)=v\overline{w}.
\end{equation}
%We identify the Heisenberg group with its image under this representation and denote it by $\mathbb{H}$.
%(Mention Heisenberg group?)

\subsection*{Metaplectic operators.} 
%Denote by $U(n)$ the unitary group in $\scrB(\C^n)$. 
Since $\U(n)$ is generated by ${\rm O}(n)$ and $\U(1)$, see~\cite[Lemma 1]{SaSch2}, we can define a 
unitary representation $g\mapsto R_g$ of $\U(n)$ by complex metaplectic operators 
$R_g\in \U(L^2(\R^n))$  by the following two assignments:
\begin{enumerate}
\item $R_gu(x)=u(g^{-1}x)$, if $g\in {\rm O}(n)\subset \U(n)$. Then  $R_g$  is a so-called shift operator;
\item $R_gu(x)=e^{i\varphi(1/2-H_1)}u(x)$, where $g={\rm diag}(e^{i\varphi},1,\ldots,1),$ and $H_1=\frac12(x_1^2-\partial_{x_1}^2)$ is the one-dimensional  harmonic oscillator. 
$R_g$ is a so-called fractional Fourier transform with respect to $x_1$. According to Mehler's formula \cite[Section 4]{H95} it is given by 
\begin{eqnarray*}\lefteqn{R_gu(x)
=\sqrt{\frac{1-i\ctg \varphi}{2\pi}}}\\
&&\times \int 
   \exp\left(
        i\left(
         (x_1^2+y_1^2)\frac{\ctg \varphi}2-\frac{x_1y_1}{\sin\varphi} 
        \right) 
       \right) u(y_1,x_2,...,x_n)dy_1.
\end{eqnarray*}
\end{enumerate} 
 Here we choose the square root so that the real part is positive. % at $\frac\pi2+\pi\Z $. 

In (ii) we may  assume that $0<\gvp<2\pi$, $\gvp\not=\pi$, as the cases $\gvp=0,\pi$ are covered by (i): For $\gvp=0$,  $R_g$ is the identity, and for $\gvp=\pi$,  it is the reflection in the first component. Moreover,  for $g=(\exp(i\pi/2), 1,\ldots, 1)$,  $R_g$ coincides with the Fourier transform in the first variable, see  the discussion on p.\,427 in \cite{H95}. 
By Theorem~7.13 in~\cite{deG1}
\begin{equation} 
\label{cr12}
R_g T_w R_g^{-1}=  T_{g w}.
\end{equation}

%
%\subsection*{The operator algebra} 
%We consider the algebra $\cA$ generated in $\scrB(L^2(\R^n))$ by the operators $R_g$, $g\in U(n)$, 
%$T_w$, $w\in\C^n$   and $A\in \Psi$. 
%Using the identities \eqref{cr12} and \eqref{Egorov} above, this is the algebra of all operators of the form 
%\begin{eqnarray}\label{def_d}%
%D= \sum_j R_{g_j}  T_{w_j} A_j,
%\end{eqnarray}
%where  $g_j\in U(n)$, $w_j\in \C^n$,  $A_j\in \Psi(\R^n)$, and the sum is finite. 
%

\section{Main Results}\label{main}

Let $E$ be a (trivial) vector bundle over $\R^n$ and  $h$ a classical Shubin symbol of order $\fm>0$ whose principal symbol is positive and scalar, i.e.~a multiple of $id_E$, positively homogeneous of degree $\fm$ for $|x|^2+|p|^2\ge 1$. Denote by $H$ the associated pseudodifferential operator acting on sections of $E$. 

%We  also introduce the sector 
%\begin{eqnarray}\label{S_theta}%
%S_\vartheta = \{z\in \C: |\arg z|\ge \vartheta\} \cup \{0\}.
%\end{eqnarray} 

For  $0<\delta\le\pi$ write $S_{\delta}$ for the sector 
\begin{eqnarray}\label{sector}% 
S_{\delta} = \{\lambda\in \mathbb C\setminus \{0\}\mid  |\arg\lambda - \pi |<\delta\}
\end{eqnarray}
and  $U_r(0)$ for the disk of radius  $r>0$ about $0\in \C$. 
Standard pseudodifferential techniques show: 

\begin{proposition}\label{prop_resolvent}
For every $\gd<\pi$,
$H-\lambda$ is invertible for all $\lambda \in S_\gd$ with  $|\lambda| $ sufficiently large.
The function $\lambda\mapsto (H-\gl)^{-1}$ is meromorphic in $S_\gd\cup U_r(0)$ for suitably small
$r>0$.
In particular, we may assume that $H-\lambda$ is invertible for all 
$\gl \in S_\gd$ after replacing $H$ by $H+c$ for suitable $c>0$.  
\end{proposition} 

We shall establish the following results:

\begin{theorem}\label{thm_resolvent} 
Fix $g\in \U(n)$, $w\in \C^n$, and a pseudodifferential operator $A\in \Psi$ of order $\ord A$. 
%Recall that $\fm>0$ is the order of $H$ and write $m=\dim_\C(\C^n)^g$. 
Choose $K\in \N$ so large that $-K\ord H +\ord A<-2n$ and fix $\gd<\pi$. Then 
\begin{enumerate}\renewcommand{\labelenumi}{(\alph{enumi})} 
\item The operator $R_gT_w A (H-\lambda)^{-K} $ is of trace class for all $\lambda$ in  $S_\gd$, $|\gl|$ sufficiently large. 

\item  The function $\lambda \mapsto \Tr(R_gT_w A (H-\lambda)^{-K}) $ is meromorphic in $S_\gd$  and near $\gl=0$. It is $O(|\gl|^{-\gve})$ with suitable $\gve>0$ for $\gl\to \infty$ in $S_\gd$. 

\item As $\gl\to \infty$ in the sector $S_\gd$ we have the expansion \eqref{resolvent_exp}.  
%\begin{eqnarray*}%
%\Tr(R_gT_w A (H-\lambda)^{-K}) 
%&\sim& \sum_{j=0}^\infty c_j (-\lambda)^{(2n+\ord A-j)/\ord H-K} \\
%&&+\sum_{j=0}^\infty (c_j' \log (-\lambda)+c_j'')(-\lambda)^{-j-K}
%\end{eqnarray*}
%with suitable coefficients $c_j, c_j', c_j''$, $j=0,1,\ldots$.  

\item If the fixed point set of the affine mapping $\mathbb{C}^n\to\mathbb{C}^n, v\mapsto gv+w$ is empty, then the trace is $O(\gl^{-\infty})$. % for all $N\in \N$. 

\item The coefficient $c_0'$ is independent of the choice of $H$ (subject to the properties stated above)  when multiplied by the factor $\ord H$.
\end{enumerate} 
\end{theorem}

%{\color{blue} 
A sketch of the proof of Theorem \ref{thm_resolvent} will be given in Section \ref{sect_resolvent}, below. 
Statements (a) and (b) follow immediately from Proposition \ref{prop_resolvent} and the fact that the embedding $\cH^{2n+\gve}(\R^n) \hookrightarrow L^2(\R^n)$ 
is trace class for every $\gve>0$. Deriving the expansion \eqref{resolvent_exp} in (c) is the crucial task. 
%In (c), the coefficients $c_j$ and $c_j'$ are `local', i.e. they are determined by the homogeneous components in the symbol expansions of $A$ and $H$, while the $c_j''$ are `global': 
%They also depend on the residual parts of the symbol. This is noteworthy, as it shows that the coefficients of $\mu^{-K-j}$, $j=0,1,\ldots$, can not simply be determined from the expansion of the symbol into homogeneous terms.  
Statements (d) and (e) will be proven along the way of showing (c). 
%} 

As a consequence, we obtain two more results. Theorems  \ref{thm_zeta} and \ref{thm_exp} follow from Theorem \ref{thm_resolvent} with the help of 
Propositions 2.9 and 5.1 in \cite{GS2}.

\begin{theorem}\label{thm_zeta}With the notation of Theorem {\rm \ref{thm_resolvent}}
assume, moreover,  that  $H-\gl$ is invertible for all $\gl\in S_{\gd}\cup U_r(0)$ for some $r>0$. Then 
\begin{enumerate} \renewcommand{\labelenumi}{(\alph{enumi})}
\item The function 
$z\mapsto \zeta_{R_gT_wA}(z):=\Tr(R_gT_wAH^{-z})$ is defined and holomorphic in the half-plane   $\{z\in \C: \ord H \re z > 2n+\ord A\}$.

\item $\zeta_{R_gT_wA}$ has a meromorphic continuation to $\C$ with at most simple poles 
in the points $(2m+\ord A-j)/\ord H$, where $m$ is the complex dimension of the
fixed point set of $g$. 
The function $\Gamma \zeta_{R_gT_wA}$ has the pole structure \eqref{trace_exp}.
%$$\Gamma(z)\zeta_{A,g,w}(z)\sim \sum_{j=0}^\infty \frac{\tilde c_j}{z-(2m+\ord A-j)/\ord H}
%+ \sum_{j=0}^\infty \Big(\frac{\tilde c_j'}{(z+j)^2} + \frac{\tilde c_j''}{z+j}\Big)
%$$
%with suitable coefficients $\tilde c_j, \tilde c_j'$ and $\tilde c_j''$. 

\item The coefficients $\tilde c_j, \tilde c_j'$ and $\tilde c_j''$ in \eqref{trace_exp}  are related to the coefficients $c_j, c_j'$ and $c_j''$ in \eqref{resolvent_exp} by universal constants. In particular, we have 
$$\Res\zeta_{R_gT_wA} = c_0' = \tilde c_0'.$$ 
 
\item If the fixed point set of the affine mapping $\mathbb{C}^n\to\mathbb{C}^n, v\mapsto gv+w$ is empty, then $\zeta_{R_gT_wA}$ has no poles.

\item $\zeta_{R_gT_wA}$ has rapid decay along vertical lines $z=c+it$, $t\in \mathbb R:$
\begin{equation}\label{eq-decay3}
 |\zeta_{R_gT_wA}(z)\Gamma(z)|\le C_N (1+|z|)^{-N},\quad \text{for all }N\ge 0, 
 \quad |\im z|\ge 1,
\end{equation}
uniformly for $c$ in compact intervals.
\end{enumerate} 
\end{theorem}

\begin{theorem}\label{thm_exp}We use the notation in Theorem {\rm \ref{thm_resolvent}} and assume additionally that $H-\gl$ is invertible for $\gl\in \overline S_\gd$ for some $\gd\in {]\pi/2, \pi]} $. Then:
\begin{enumerate} \renewcommand{\labelenumi}{(\alph{enumi})}
\item The function $t\mapsto \Tr(R_gT_wA\exp(-tH)) $ is defined for all $t>0$. 
\item As $t\to 0^+$ it has the asymptotic expansion \eqref{exp_exp}
%$$\Tr( R_gT_wAe^{-tH)})  \sim \sum_{j=0}^\infty \tilde c_j t^{(j-m-\ord A)/\ord H} + 
%\sum_{j=0}^\infty (-\tilde c'_j\log t + c_j'')t^j$$
with the same coefficients as for the zeta function in Theorem {\rm \ref{thm_zeta}}. 
\end{enumerate}
\end{theorem}

\subsection*{The noncommutative residue} 
Let $G$ be a  subgroup of $\C^n\rtimes \U(n)$. 

\begin{definition}\label{def_ncr}
For $w_0\in \C^n$, $g_0\in \U(n)$ such that $(w_0,g_0)\in G$ we define the noncommutative residue 
$\res_{\skp{(w_0,g_0)}} $
localized at the conjugacy class $\skp{(w_0,g_0)}$ 
in $G$ by 
\begin{eqnarray*}\label{def_res}
\lefteqn{\res_{\skp{(w_0,g_0)}} D = \ord H
\sum_{(w,g)\in \skp{(w_0,g_0)}}\Res \zeta_{A,g,w}(z)}\\
&=& \ord H \sum_{(w,g)\in \skp{(w_0,g_0)}} c_0' ( R_gT_wA)
=\ord H \sum_{(w,g)\in \skp{(w_0,g_0)}}\tilde c_0' ( R_gT_wA)
\end{eqnarray*} 
with the coefficients $c_0'$ and $\tilde c_0'$ introduced in Theorem
\ref{thm_resolvent} and Theorem \ref{thm_zeta}, respectively, 
and the order $\ord H$ of $H$. The expression is independent of $H$ by Theorem \ref{thm_resolvent}(e). 
\end{definition} 
While the definition in terms of the zeta residue or the coefficient in the heat trace expansion is more common, that in terms of the coefficient in the resolvent expansion is slightly more general, since we need less assumptions on the auxiliary operator to define it.

One can compute the residue explicitly. Consider the case of diagonal $g$: 
\begin{equation}
\label{eq-diag2}
g={\rm diag}\Bigl(\underbrace{e^{i\varphi_1},...,e^{i\varphi_{m_1}}}_{m_1},\underbrace{i,...,i}_{m_2},\underbrace{-i,...,-i}_{m_3},\underbrace{-1,...,-1}_{m_4},\underbrace{1,...,1}_{m_5}\Bigr)\in \U(n),
\end{equation}
where $\varphi_j\in {]-\pi,\pi[}\setminus  \pi\mathbb{Z}/2$ and   $m_5=\dim_\C(\mathbb{C}^n)^g$.   
For notational convenience we will also write  $\varphi_j=\pi/2$ for 
$j=m_1+1, \ldots, m_1+m_2$ and $\varphi_j = -\pi/2$ for $j= m_1+m_2+1,\ldots, m_1+m_2+m_3$. 

\begin{theorem}\label{residue}
For  $D=R_gT_wA$ with $g$ as above, 
$w=a-ik\in \C^n$ and  $A\in \Psi^{-2m_5}$ we find the explicit formula: 
\begin{eqnarray}%
\nonumber
\lefteqn{\res_{\skp{(w,g)}} D = 
 (2\pi)^{-n+m_2+m_3} 
\prod_{j=1}^{m_1}  \sqrt{1+i\tg\varphi_j}  %AS
\nonumber}\\
&&\times  
   \prod_{j=1}^{m_1+m_2+m_3}
 e^{\frac i4\ctg(\varphi_j/2)(k_j^2+a_j^2)}
 \int_{\mathbb S^{2m_5-1}}\tr_E a_{-2m_5} (\theta)\,dS \label{eqn_res}	
\end{eqnarray}
where we choose the square root in the right half plane, 
%ASthe $\delta_j$ are given by  %\eqref{lambdaj} and \eqref{deltaj}, below, 
$\mathbb S^{2m_5-1}$ is the unit sphere in the $+1$-eigenspace of the matrix 
$g$, and  $a_{-2m_5}$ is the homogeneous component of degree $-2m_5$ in the asymptotic expansion of the symbol of $A$.

In particular, the residue vanishes for  $A\in \Psi^\fa$ with $\fa<-2m_5$. 
For $\fa>-2m_5$, additional terms containing derivatives of components $a_{\fa-j}$, $\fa-j>-2m_5$ enter. For details see Section \ref{sect_residue}. 
\end{theorem}

\begin{remark}
\label{3.7}
For an arbitrary element  $g\in \U(n)$, there exists a unitary $u\in \U(n)$ such that 
$g=ug_0u^{-1}$, where $g_0$ is diagonal and unitary. Then 
\begin{eqnarray*}%\lefteqn{
\Tr(R_gT_wA H^{-z})
%=\Tr(R_uR_{g_0}R_u^{-1}T_wA H^{-z})}\\
%&=&\Tr( R_{g_0}R_u^{-1}T_w(R_u R_u^{-1})A (R_u R_u^{-1})H^{-z}R_u)\\
%&=&\Tr( R_{g_0}(R_u^{-1}T_wR_u)(R_u^{-1} A  R_u)  (R_u^{-1}H^{-z}R_u))\\
&=&\Tr( R_{g_0}T_{w'} A' (H')^{-z}),
\end{eqnarray*}
where $A'=R_u^{-1} A  R_u$ is a Shubin type operator by Egorov's theorem, $T_{w'}=R_u^{-1}T_wR_u$ with $w'=u^{-1}w$, and  $H'=R_{u^{-1}}HR_{u}$
is an admissible auxiliary operator of the same order, so that  Theorem \ref{residue} applies. 

In case $H=H_0$ is the harmonic oscillator of \eqref{def_ham}, we even have $H'=H_0$, since $H_0$ commutes with $R_u$. 
\end{remark} 

\begin{theorem}\label{3.8} 
For every choice of a conjugacy class $\skp{(w_0,g_0)}$, the  functional $\res_{\skp{(w_0,g_0)}}$ is  a trace.
\end{theorem}  

\begin{proof}
We use a modification of an argument given by  Dave \cite{D1}.
By linearity it is sufficient to consider two elements $B_1= R_{g_1}T_{w_1}A_1$ and $B_2=R_{g_2}T_{w_2}A_2$. Equations \eqref{eq-comm1} and \eqref{cr12}
imply that 
\begin{eqnarray*}
%\lefteqn{
R_{g_1}T_{w_1}R_{g_2}T_{w_2} =R_{g_1} R_{g_2} T_{g_2^{-1}w_1}T_{w_2}= 
e^{i \im(g_2^{-1}w_1,w_2)} R_{g_1g_2}T_{g_2^{-1} w_1+w_2},
\end{eqnarray*}
so we want to show that 
$$0=\res_{\skp{(g_2^{-1}w_1+w_2,g_1g_2)}}([B_1,B_2])=2\Res([B_1,B_2]H_0^{-z}) .$$
We have chosen here $H_0$ as auxiliary operator; we know from %ES reference corrected
Theorem \ref{thm_resolvent}(e) that the result does not depend on the choice. 

For large $\re(z)$, the operators $B_1B_2H_0^{-z}$, $B_2B_1H_0^{-z}$ and 
$B_1H_0^{-z} B_2$ are trace class operators on $L^2(\R^{n})$ and therefore 
$$\Tr ([B_1,B_2]H_0^{-z}) = \Tr(B_1[B_2,H_0^{-z}]).$$
The fact that $H_0$ and hence $H_0^{-z}$ commutes with $R_{g_2}$  implies that  
\begin{eqnarray*}%
[B_2,H_0^{-z}] &=& [R_{g_2} T_{w_2}A_2,H_0^{-z}] = 
R_{g_2} [T_{w_2}A_2,H_0^{-z}]  +[R_{g_2}, H_0^{-z}]T_{w_2}A_2\\
&=&
R_{g_2}T_{w_2} [A_2,H_0^{-z}] + R_{g_2} [T_{w_2},H_0^{-z}]A_2.
\end{eqnarray*}
Let us show that the contribution to the residue from both terms vanishes.

For the first we study the asymptotic expansion of the symbol of the commutator 
$[A_2,H_0^{-z}]$.
The principal symbol of $H_0^{-z}$ is proportional to  $h(x,p)^{-z} = (|x|^2+|p|^2)^{-z}$, while the terms of degree $-2z-j$ in the asymptotic expansion are linear combinations of expressions of the form 
$$(-z)(-z-1)\cdots(-z-k+1) d_{k,j} (x,p) h(x,p)^{-z-k}, \quad 1\le k\le 2j,$$
where $d_{k,j}$ is a polynomial in $x$ and $p$ of degree $2k-j$,
see e.g. \cite[Equation (25)]{Seeley67}.
Hence for arbitrary $N$, the commutator can be written in the form 
$$[A_2, H_0^{-z} ]= \op (q_N(z)+r_N(z)) ,$$
where 
\begin{itemize}
\item $q_N(z)$ is a linear combination of terms of the form 
$$z \cdots(z+k-1) \partial_x^\ga \partial_p^\gb a(x,p) d(x,p) (|x|^2+|p|^2)^{-k-z}$$ 
with suitable derivatives of $a$, polynomials $d$ in $x$ and $p$ and $k\ge 1$. 
\item $z\mapsto r_N(z)$ is a holomorphic family of symbols of order $-N$. 
%in a half-plane including $z=0$.  
\end{itemize} 
We can therefore write, for suitable $N'$, 
$$[A_2, H_0^{-z} ]=\sum_{k=1}^{N'} z^k C_k H_0^{-z} + R_N(z).$$ 
with Shubin pseudodifferential operators $C_k$ of order $\le \ord A_2$ and a family 
$z\mapsto R_N(z)$ of trace class operators that is holomorphic near $z=0$. 
Since $\Tr(BH^{-z})$ has at most a simple pole in $z=0$ for any operator $B$ in the 
calculus, we conclude that the contribution from the terms under the summation is zero. 
On the other hand, if  $N$ is  sufficiently large, then $\Tr(BR_N(z))$ will also be holomorphic in $z=0$ and therefore not contribute to the residue. 
This shows that the contribution from the first term is zero.  

For the second, we recall that 
$$
H_0^{-z} = \frac1{2\pi i}\int_{\mathscr C} \gl^{-z} (\lambda -H_0)^{-1} \, d\gl,$$
where $\mathscr C$ is a contour surrounding the spectrum of $H_0$; specifically 
we can choose here $\mathscr C= \{re^{i\vartheta}: \infty>r\ge0\} \cup \{re^{-i\vartheta}:0\le r<\infty\}$ for some angle $\vartheta>0$. 
The operator $T_{w_2}$ is bounded on $L^2(\R^n)$ and
\begin{eqnarray*}%
T_{w_2}^{-1} (\gl-H_0)^{-1} T_{w_2}&=&(T_{w_2}^{-1}  (\gl-H_0)T_{w_2})^{-1}
= (\gl-T_{w_2}^{-1} H_0T_{w_2})^{-1},
%\\&=&  (\gl-H_0+R)^{-1}
\end{eqnarray*}
so that $T_{w_2}^{-1}H_0^{-z}T_{w_2} = H_{w_2}^{-z} $ with 
$H_{w_2} =  T_{w_2}^{-1}H_0T_{w_2}$. 
We conclude that 
$$[T_{w_2},H_0^{-z}] = T_{w_2} (H_0^{-z} -T_{w_2}^{-1} H_0^{-z} T_{w_2}) = T_{w_2} (H_0^{-z} - H_{w_2}^{-z}) . $$
As a consequence, 
\begin{eqnarray*}\lefteqn{
\Tr( B_1 R_{g_2} [T_{w_2},H_0^{-z}]A_2) = \Tr(A_2 B_1 R_{g_2} [T_{w_2},H_0^{-z}]) }\\
&=& \Tr(A_2 B_1 R_{g_2} T_{w_2}H_0^{-z}) - \Tr(A_2 B_1 R_{g_2} T_{w_2}H_{w_2}^{-z}) .
\end{eqnarray*}
Now the principal symbol of $H_{w_2}$ coincides with that of $H_0$; it is therefore also an admissible auxiliary operator. Since we know that the residue is independent of the choice of the auxiliary operator, we see that the difference of the residues must vanish. 
\end{proof}  
 
\section{The Resolvent Expansion}\label{sect_resolvent} 

Recall that, by assumption, $H$ is a classical Shubin type operator of order $\ord H = \fm\in \Z_{>0}$.
Its principal symbol $h_\fm=h_\fm(x, p)$ is scalar and, moreover, strictly positive  and  homogeneous of  degree $\fm$ in $(x,p)$ for $|(x,p)|\ge1$. .
We write $-\lambda = \mu^{\fm}$ for $\mu$ in the sector  
$$S= S_{\theta/\fm}=\{z\in \C\setminus \{0\}: |\arg z|\le (\pi-\theta)/\fm\}$$ 
so that $h_\fm(x,p)+\mu^\fm$ is invertible for all $\mu \in S$. 

Given a Shubin pseudodifferential operator $A$ of order $\fa$, choose $K\in \N$ so large that $\fa-K\fm<-2n$, so that the operators 
$A(H+\mu^\fm)^{-K}$ and $R_gT_wA(H+\mu^\fm)^{-K}$ are of trace class on $L^2(\R^n)$ for every $\mu$ in the sector $S$, $|\mu|$ sufficiently large.

Denote by $q(y,p;\mu)$ the (complete) symbol of $A(H+\mu^\fm)^{-K}$ in $y$-form,  
i.e. 
\begin{eqnarray}%\label{}\lefteqn{}\\
\Big(A(H+\mu^\fm)^{-K} \Big)u(x) = (2\pi)^{-n} \int e^{i\skp{x-y,p}} q(y,p;\mu) u(y) dy dp. 
\end{eqnarray}

For $w=a-ik$ and a diagonal element $g$ as in \eqref{eq-diag2},
  consider the function 
\begin{eqnarray*}%
\mu \mapsto \Tr(R_gT_w A(H+\mu^\fm)^{-K})
\end{eqnarray*}

Our first step and the basis of the subsequent analysis is the proof of the following proposition.

\begin{proposition}\label{prop_trace}
Up to terms that are $O(|\mu|^{-\infty})$ as $\mu\to \infty $ in $S$, \begin{eqnarray*}%
 \Tr(R_gT_w A(H+\mu^\fm)^{-K}) = 
 C_{\rm res} \int e^{i\phi(u,v)}\tr_E q(B(u,v) -b_0, y,\eta;\mu)\, dud\underline{v}dyd\eta.
\end{eqnarray*}
Here $C_{\rm res}$ is the constant in front of the integral in Theorem \ref{residue} above, $\tr_E$ is the trace in the vector bundle $E$,
$$(x_1,p_1, \ldots x_{n-m_5},p_{n-m_5})= B(u,v)-b_0,
\quad  u,v\in\mathbb{R}^{n-m_5}
$$ 
is an affine linear change of coordinates. 

As a mapping $ (u_j,v_j)\mapsto(x_j,p_j) $, $j=1, \ldots, n-m_5$ the matrix $B$ is a diagonal block matrix with entries $\frac1{\sqrt2}\mat{1&1\\-1&1}$ for  $j\not=m_1+1,\ldots,m_1+m_2+m_3$, for the latter it is the identity. 
In particular, $B$ is orthogonal. 
Moreover, 
$$b_0 = (u_{0,1}, v_{0,1}, \ldots, u_{0,n-m_5}, v_{0,n-m_5})$$  
with 
$$(u_{0,j},v_{0,j} ) 
\begin{cases} \text{as in } \eqref{ujvj},\quad  j=1,\ldots, m_1 
\\
 =((a_j-k_j\sin\gvp_j)/2,0), \quad j=m_1+1,\ldots, m_1+m_2+m_3\\
=(a_j/2,k_j/2), \quad j=m_1+m_2+m_3+1,\ldots, n-m_5
\end{cases} 
$$

Finally $y$ and $\eta$  encompass the variables,
$y=(x_{n-m_5+1},\ldots x_n)$, $\eta=(p_{n-m_5+1},\ldots p_n)$.

The notation `$d\underline v$' indicates that the integration is not over $v_{m_1+1}, \ldots v_{m_1+m_2+m_3}$; for these variables, $q$ is evaluated at the  points
$v_j = \sin\gvp_j(u_j+\frac{a_j-k_j\sin\gvp_j}2)-k_j$.
%, with `$+$' for $j=m_1+1, \ldots ,m_1+m_2$ and `$-$' else, see \eqref{ujvj}, below.  

The phase $\phi$ is given by 
\begin{eqnarray}\label{phase_phi}%
\phi(u,v)&=&
-\sum_{j=1}^{m_1} (\lambda^-_ju_j^2 +\lambda^+_jv_j^2)-\sum_{j=m_1+1}^{m_1+m_2}u_j^2    
+\sum_{j=m_1+m_2+1}^{m_1+m_2+m_3}u_j^2  \nonumber \\        
&&-\sum_{j=m_1+m_2+m_3+1}^{n-m_5}(-u_j^2+v_j^2)\label{eq_phase}
\end{eqnarray}
with $\lambda_j^\pm$ defined in \eqref{lambdaj}, below. 
\end{proposition} 

That $b_0$ has no entries for $j=n-m_5+1,\ldots,n$ is due to the fact that the expansion is 
$O(|\mu|^{-\infty})$ whenever $a_j\not=0$ or $k_j\not=0$ for some $j>n-m_5$.

\begin{proof}
The proof is lengthy; here is a sketch: 

{\em Step} 1. A computation very similar to that for Equations (76) and (77) in the proof of 
\cite[Theorem 5]{SaSch3} with explicit constants shows that the trace is $O(\mu^{-\infty}) $
unless $a_j=k_j=0$ for $j>n-m_5$ and then
\begin{eqnarray*}
\lefteqn{
\Tr(R_gT_wA(H+\mu^\fm)^{-K})}\\ 
&=&  (2\pi)^{-n+m_2+m_3} e^{-i\skp{a,k}/2}
\prod_{j=1}^{m_1}
\sqrt{1+i\tg\varphi_j}
\\ 
&&\times \int e^{i\phi_3(x,p')} \tr_Eq(x,p',p'';\mu)_{|p''_j = \pm x_j-k_j}\, dp'dx.
\end{eqnarray*}
Here, the notation $ q(x,p',p'';\mu)_{|p''_j = \pm x_j-k_j}$ indicates that we 
evaluate $p''_j$ at $x_j-k_j$ for $j=m_1+1,\ldots, m_1+m_2$ and at $-x_j-k_j$ 
for $j=m_1+m_2+1,\ldots, m_1+m_2+m_3$, and the phase $\phi_3$ is given by
\begin{eqnarray}\label{phase3}
\lefteqn{\phi_3(x,p')=
\sum_{j=1}^{m_1}
     \Bigl(
        x_j^2\Bigl(\frac{\ctg\varphi_j}2-\frac{1}{ \sin 2\varphi_j }\Bigr)
       +x_jp_j\Bigl(\frac{1}{\cos\varphi_j}-1 \Bigr)}\\
       &&-p_j^2\frac{\tg\varphi_j}2  
       +x_j\frac{k_j}{\cos\varphi_j}
       -p_j\left( a_j+k_j\tg\varphi_j\right)-\frac{ \tg\varphi_j}{2}k_j^2
      \Bigr)\nonumber  \\
&&-\sum_{j=m_1+1}^{m_1+m_2+m_3}(x_j+a_j)\left(\frac{x_j}{\sin\varphi_j}-k_j\right)   \nonumber  \\  
&&-\sum_{j=m_1+m_2+m_3+1}^{n-m_5}( 2x_jp_j +a_j p_j+k_jx_j ).\nonumber
\end{eqnarray}

Noting that 
\begin{eqnarray*}%
\ctg\varphi_j -\frac1{\sin\varphi_j\cos\varphi_j }
&=& \frac{\cos^2\varphi_j -1}{\sin\varphi_j\cos\varphi_j }
=-\tg\varphi_j\text{ and}\\
\frac1{\cos\varphi_j}-1&=& \tg\varphi_j \ \frac{1-\cos\varphi_j}{\sin\gvp_j} 
\end{eqnarray*}
we see that 
\begin{eqnarray*}
\phi_3(x,p) &=& 
-\sum_{j=1}^{m_1} \frac{\tg\gvp_j}2
\Bigg(x_j^2\ 
-2x_jp_j\Big(\frac{1-\cos\varphi_j}{\sin\gvp_j}\Big)
+p_j^2 
+2x_j\frac{k_j}{\sin \varphi_j}\\
&&+2p_j\Big(%ES -
\frac{a_j}{\tg \gvp_j} 
+k_j\Big)+k_j^2\Bigg)
-\sum_{j=m_1+1}^{m_1+m_2+m_3}(x_j+a_j)\left(\frac{x_j}{\sin\varphi_j}-k_j\right)   \\  
&&-\sum_{j=m_1+m_2+m_3+1}^{n-m_5}( 2x_jp_j +a_j p_j+k_jx_j ).
\end{eqnarray*}

{\em Step} 2. We next make several changes of coordinates of the form 
$$\binom{x_j}{p_j} = B_j\binom{x_j}{p_j}- \binom{u_{0,j}}{v_{0,j}}, \quad j=1,\ldots, n-m_5.$$ 

For  $j=1,\ldots,  m_1$ we let $B_j= \frac1{\sqrt2}\mat{1&1\\-1&1}$,  
\begin{align}
u_{0,j} &= \frac{2\gb_j-\gg_j\ga}{4-\ga_j^2}, \quad\text{and }\quad   v_{0,j} =  \frac{2\gg_j-\gb_j\ga_j}{4-\ga_j^2}\quad\text{with} \label{ujvj}
\\
\ga_j &= -2\frac{1-\cos\gvp_j}{\sin\gvp_j}, \quad  %\label{alphaj}\\
\gb_j =\frac{2k_j}{\sin\gvp_j},\quad %label{betaj}\ 
\gg_j = \frac{2a_j}{\tg \gvp_j}+2k_j.\label{gammaj}
\end{align}

For  $j=m_1+1,\ldots ,m_1+m_2+m_3$, where $\gvp_j=\pm \frac\pi2$, we let
$u_j =x_j+\frac{a_j-k_j\sin\gvp_j}{2}= x_j+\frac{a_j\mp k_j}{2}$, with $+$ for $j=m_1+1,\ldots ,m_1+m_2$, and  $-$ else. Recall that here, $p_j = \pm x_j-k_j$. 

Finally for $j=m_1+m_2-m_3+1,\ldots, n-m_5$, we let $B_j$ be as above, $u_{0,j}= a_j/2$ and 
$v_j=k_j/2$. 

This leads to the expression for $\Tr(R_gT_w A(H+\mu^\fm)^{-K})$ in the statement for the phase \eqref{phase_phi} with  
\begin{eqnarray}\label{lambdaj}%
\gl_j^\pm =\frac{\tg\varphi_j}4 (2\pm \ga_j )= \frac{1}2  \frac{\sin \gvp_j \mp(1-\cos\gvp_j)}{\cos\gvp_j},\ j=1,\ldots, m_1
\end{eqnarray}
(note that $\gl_j^\pm\not=0$ for  $j=1,\ldots, m_1$) 
and the constant $C_{\rm res}$.

This is what we had to show. 
\end{proof}

\subsection*{A Simplification: Omitting the Heisenberg-Weyl Operators}

%We omit the operators $T_w$. 
If  $a_j=k_j=0$ for all $j$, the above formula becomes 
\begin{eqnarray*}%
&&\Tr(R_gA(H+\mu^{\fm})^{-K})
=  (2\pi)^{-n+m_2+m_3} 
\prod_{j=1}^{m_1}  
\sqrt{1+i\tg\varphi_j}\\
&&\quad\times\int e^{i\phi_6(u,v)} \tr_E q(B\mat{u\\v}-b_0,x'''',p'''';\mu)\, dudvdx''''dp''''
\end{eqnarray*}
with %ES a suitable orthogonal matrix $B$ and coordinates 
$u,v$ subsuming the 
above coordinates $u',u'',u''',v',v'',v'''$, where $v''$ is evaluated at 
$v''_j= \sin\gvp_j (u''_j-\frac{a_j-k_j\sin\gvp_j}{2})-k_j$ and   
\begin{eqnarray*}%
\phi_6 &=&
-\sum_{j=1}^{m_1} (\lambda^-_ju_j^2 +\lambda^+_jv_j^2)-\sum_{j=m_1+1}^{m_1+m_2}u_j^2    
+\sum_{j=m_1+m_2+1}^{m_1+m_2+m_3}u_j^2   \\        
&&-\sum_{j=m_1+m_2+m_3+1}^{n-m_5}(-u_j^2+v_j^2).
\end{eqnarray*}

\subsection*{Sketch of the Proof of Theorem \ref{thm_resolvent}} We proceed in several steps. 
%In order to make the following computations more transparent, we start with an outline of the steps we will take. 

%ES Changed sign of b_0 and b (now \tilde b_0) 
\subsubsection*{Step 1. Notation}
In order to keep the notation simple, we shall rewrite the expression found
in Proposition \ref{prop_trace}, namely 
\begin{eqnarray*}%
 \Tr(R_gT_w A(H+\mu^\fm)^{-K}) = 
 C_{\rm res} \int e^{i\phi(u,v)}\tr_Eq(B\mat{u\\v} -b_0, y,\eta;\mu)\, dud\underline{v}dyd\eta.
\end{eqnarray*}
in the form 
\begin{eqnarray}\label{eq:simple} 
\Tr(R_gT_w A(H+\mu^\fm)^{-K}) = 
 C_{\rm res} \int_{\R^{n'}} e^{i\phi(w)}\tilde q(\tilde Bw-\tilde b_0;\mu)\, dw.
\end{eqnarray}
Here, 
\begin{itemize} 
\item $w=(w_1,\ldots,w_{n'})$ where $n'=2n-m_2-m_3$ denotes the number of variables used in the integration, 
\item 
$\tilde q$ is $\tr_Eq$ with the variables  $v_j$, $j=m_1+1,\ldots,m_2+m_3 $, evaluated at 
$v_j =  \sin\gvp_j (u_j-\frac{a_j-k_j\sin\gvp_j}{2})-k_j$ (recall that for these $j$, 
we have $\gvp_j=\pm \pi/2$)
\item $\tilde B\in \mathrm O({n'})$ and $\tilde  b_0\in \R^{n'}$ are made up from the corresponding components of $B$ and $b_0$. 
\item 
$\phi(w) = \sum_{j=1}^{n'}\tilde \lambda_j w_j^2 $, where the $\gl_j$ are determined by the coefficients in the representation of $\phi_6$, i.e. the  
$\gl_j^\pm$, $j=1,\ldots, m_1$, the coefficients $-1$ and $+1$ of $u_j^2$, $j=m_1+1,\ldots, m_1+m_2+m_3$, the coefficients $+1$ of $u_j^2$  and $-1$ of $v_j^2$ for $j=m_1+m_2+m_3+1,\ldots , 2n-m_5$,   and where we complement the summation by letting  
$\tilde \lambda_j=0$ for the variables corresponding to $(u_j,v_j)$, $j>2n-2m_5$.
\end{itemize} 

\subsubsection*{Step 2. Reduction of the integrand}
Instead of considering the full symbol $\tilde q(\tilde Bw-\tilde b_0;\mu)$ in \eqref{eq:simple}, 
it will suffice to study the homogeneous components in the expansion of the symbol of $\tilde q$. 

This is a consequence of the following  results that can be shown to hold more generally in the Grubb-Seeley calculus (we here use the notation from the appendix). 
We first consider the effect of the affine linear transformation $w\mapsto \tilde Bw-\tilde b_0$. 

\begin{lemma} \label{4.5}
Let $q(w; \mu)\in S^{m,d}(\R^{n'};S)$, $B\in \cL(\R^{n'})^{-1} $ and $b\in \R^{n'}$.  
Define $ q_1(w;\mu) = q(Bw+b;\mu)$ and $q_2(w;\mu)  = p(w+b;\mu)$.  
Then $q_1$ and $q_2$ are elements of $S^{m.d}(\R^{n'};S)$. After a finite Taylor expansion 
\begin{eqnarray*}%
p(w+b;\mu) &=& \sum_{|\ga|< N} \frac{b^\ga}{\ga!} \partial^\ga_w p(w;\mu) +r_N(w;\mu),
\end{eqnarray*}
the remainder  $r_N$ is in $S^{m-N,d}(\R^{n'}; S)$.
\end{lemma}

Moreover, the residual terms have the right behavior:   

\begin{lemma} \label{5.5}
Let $r\in S^{-N,-K\fm}$ and $n_0\in \N$ such that $N>2n+n_0$. Then 
\begin{eqnarray}\label{4.1}%
\int e^{i\phi(w)} r(w;\mu)\, dw =\sum_{k=0}^{n_0-1} c_k \mu^{-K\fm-k} + 
O(\mu^{-K\fm-n_0})  
\end{eqnarray}
with explicitly given coefficients $c_k$.  
\end{lemma}

These terms furnish the contributions to the coefficients $c_j''$. 
In fact, as pointed out by Grubb and  Seeley \cite[p.501]{GS1}, the $c_k$ are zero unless $k$ is a multiple of $\fm$, 
so that only integer powers of $-\gl =\mu^\fm$ show up. This follows from the parametrix construction, see Theorem \ref{GrubbS2}, below. 

\subsubsection*{Step 3. Focusing on the homogeneous terms.}
It is actually sufficient to study expressions of the form 
\begin{eqnarray}\label{simplified_integrand}%
%\int_{\R^{n'}} e^{i\phi(w)}
d(w) (h_\fm(w) + \mu^\fm)^{-K-\ell},
% \, dw,
\end{eqnarray}
where $\ell\in \N_0$ and $d$ is a Shubin symbol, which is homogeneous of  some
degree $\fd= \fa+ \ell \fm-j$, $j\in \N_0$,  for $|w|\ge1$. 
This is a corollary of the following proposition: 
\begin{proposition}\label{4.2}
{\rm (a)} For each $N\in \N$ the components of the complete right symbol $q= q(y,p;\mu)$ 
of  $A(H+\mu^\fm)^{-K}$ can be written as a finite sum of terms of the form 
\begin{eqnarray}\label{4.2.1}%
d(y,p)(h_\fm(y,p) +\mu^\fm)^{-K-\ell} ,
\end{eqnarray}
where $d$ is homogenous in $(y,p)$ of degree $\fd=\fa+\ell\fm-j$ for $j=0,\ldots, N-1$, 
plus a remainder in $S^{-N,-K\fm}\cap S^{-N-K\fm,0}$. 

{\rm (b)} 
Each $d$ is a polynomial in derivatives of the homogeneous components in the symbol expansions of $a$ and $h$. If $\ell=0$ in {\rm (a)}, then $d$ can only be one of the homogeneous components of $a$. 

{\rm (c)} The statement corresponding to {\rm (a)} is also true for the symbol $\tilde q$. Here,
$d$ is additionally scalar, since we take the fiber trace in $E$. 

{\rm (d)} 
In the expansion of $\tilde q(Bw-\tilde b_0;\mu)$ into homogeneous components, the only possible coefficient functions of $(h_\fm(\tilde Bw)+\mu^\fm)^{-K}$ are of the form 
$$\frac{(-\tilde b_0)^\ga}{\ga!}\partial_w^\ga \tr_E a_{\fa-j}(\tilde B w)$$
for suitable $j\in \N_0$ and a multi-index $\ga$.
\end{proposition} 

So we are led to  considering  integrals  of the form
\begin{eqnarray}\label{eq:int}%
\int_{\R^{n'}} e^{i\phi(w)}  d(w) (h_\fm(w)+\mu^\fm)^{-K-\ell} \, dw
\end{eqnarray}
with $\ell\in \N_0$ and  $d$ homogeneous of 
degree $\fd$ for $|w|\ge 1$.

Similarly as Grubb and Seeley in their analysis of weakly parametric symbols 
\cite[Theorem 2.1]{GS1}, we split the integral over $\R^{n'}$ into the integrals 
over $\{|w|\le 1\}$, $\{1 \le|w|\le |\mu|\}$, and $\{|w|\ge |\mu|\}$. 
It is not difficult  to see that  the integral over $|w|\le 1$ 
produces  powers $\mu^{-(K+\ell+j)\fm}$, $j\in \N_0$, using a binomial series expansion of the integrand. These terms also contribute to the terms  $c_j''(-\gl)^{-K-j}$ in \eqref{resolvent_exp}, even though they are 'local'.

For the other two integrals we introduce polar coordinates and use the fact that  
$\phi$ is 2-homogeneous and  
$d$ is homogeneous of degree $\fd$.

\subsubsection*{Step 4. The integral  over $|w|\ge |\mu|$} 
Letting $w/|\mu|=v=r\gt$ we obtain
\begin{eqnarray}%\label{4.7.2}
\lefteqn{
\int_{|w|\ge |\mu|} e^{i\phi(w)} d(w) (h_\fm (w)+\mu^\fm)^{-K-\ell} \, dw}\nonumber\\
&=&|\mu|^{\fd-\fm(K+\ell)+n'} \int_{|v|\ge 1} e^{i\phi(v)|\mu|^2} d(v)  
(h_\fm(v) + (\mu/|\mu|)^\fm)^{-K-\ell} \, dv\nonumber\\
&=&|\mu|^{\fd-\fm(K+\ell)+n'} \int_1^\infty I(\hbar, r, \go)_{|\hbar=(|\mu|^2r^2)^{-1}}
r^{\fd+n'-1}\, dr
\label{4.7.3}
\end{eqnarray}
with 
\begin{eqnarray*}%
I(\hbar, r,\go) = \int_{\mathbb S^{n'-1}} e^{i\phi(\gt)/\hbar}d(\gt) 
(h_\fm(r\gt)+\go^\fm)^{-K-\ell}  \, dS(\gt).\label{4.7.4}
\end{eqnarray*}

The stationary phase approximation, see e.g. \cite[Theorem 3.16]{Z}, yields an expansion 
for this integral:  
Since $\phi(\theta)=\sum_{j=1}^{n'}   \tilde \lambda_j\theta^2_j$, the stationary 
points on the sphere are the unit length eigenvectors of the matrix 
$\text{\rm diag}(\lambda_1, \ldots, \lambda_{n'})$. We denote by $\{\mu_l\}$ the set of {\em different} $\lambda_j$ and by $\gk_l$ the multiplicity of $\mu_l$. 
To fix the notation let us assume that 
$\mu_1=0$ so that, by \eqref{eq-diag2},  $\varkappa_1=2\dim _\C \mathbb{C}^n_g$ is the real dimension of the fixed point set of $g\in \U(n)$.
% (this follows from \eqref{phase3} and \eqref{eq-diag2}).
The set of stationary points therefore
is the disjoint union of the spheres $\mathbb S^{\gk_l-1}$ in the eigenspaces for $\mu_l$. 
%The stationary phase approximation therefore  shows the following
Hence we  find: 

For $N\in \N$, 
\begin{eqnarray}\lefteqn{\nonumber
I(\hbar, r,\go)}\\ 
&\mbox{\;\;\;\;}=& \sum_{J=0}^{N-1}
\hbar^{J+\frac{n'-\gk_l}2} 
\sum_le^{\frac i\hbar \mu_l}
\int_{\mathbb S^{\gk_l-1}}A_{2J}(\theta', \partial_{\theta'})\left( d(\theta) 
(h_\fm(r\gt)+\go^\fm)^{-K-\ell}\right)\, dS
\label{4.7.10}\\
&&+ O\Big(\sum_l\hbar^{N+\frac{n'-\gk_l}2} \sum_{|\alpha|\le 2N+n'-\gk_l+1} \sup |\partial^\alpha_{\theta'}\big(d(\theta)(h_\fm(r\gt)+\go^\fm)^{-K-\ell}\big)|\Big)\label{7.2} \nonumber
\end{eqnarray}
with differential operators $A_{2J}$ of order $2J$ in the variables $\theta'$ transversal to the fixed point set. 

We insert this expression for $I(\hbar, r,\go)$ with $\hbar = r^{-2}|\mu|^{-2}$ into the integral over $r$.
The contribution of the remainder term turns out to be  a low power of $|\mu|$ as desired. 
The terms in the expansion, however, give us expressions of the form 
\begin{eqnarray}
\label{expansion-terms}
%\lefteqn{
e^{i\mu_l|\mu|^2} |\mu|^{-(K+\ell)\fm -j} g_j(\mu/|\mu|)
\end{eqnarray}
for $j\in \N_0$ and smooth functions $g_j$.
This is not quite what we would like to obtain:   
There should be no exponential terms with $\mu_l\not=0$ and, moreover, we would like to have expressions in $\mu$ (and its powers), not in $|\mu|$ and $\mu/|\mu|$.  We shall address this problem in Step 6.

\subsubsection*{Step 5. The integral over $1\le |w|\le |\mu|$} 
This is the part with the most delicate analysis. 
We first expand $(h_\fm+\mu^\fm)^{-K-\ell}$ 
into a finite number of powers of $h_\fm\mu^{-\fm}$ and a remainder term. 
Proceeding similarly as in Step 4, we have to treat terms of the form 
\begin{eqnarray}
\mu^{-(K+\ell)\fm- k\fm} 
\int_{1}^{|\mu|} \int_{\mathbb S^{n'-1}}e^{ir^2\phi(\gt)}d(\gt) h_\fm(\gt)^k \, dS
\, r^{\fd+\fm k+n'-1}\, dr
\label{4.12.4}
\end{eqnarray}
and the corresponding integrals over the remainder terms. 
An application of the stationary phase approximation with $\hbar = r^{-2}$ 
for the integral over the sphere  results in terms 
\begin{eqnarray}%
\mu^{-(K+\ell)\fm- k\fm}  c_{l,J}
\sum_l \sum_{J=0}^{N-1} 
\int_{1}^{|\mu|}e^{ir^2\mu_l}
\, r^{\fd+\fm k-2J+\gk_l-1}\, dr
\label{4.14.1}
\end{eqnarray}
with 
\begin{eqnarray}
\label{4.14.1a}
%\lefteqn{
c_{l,J} = \int_{\mathbb S^{\gk_l-1}}A_{2J,l}(\theta',\partial_{\gt'})(d(\gt)h_\fm(\gt)^k)\, dS
\end{eqnarray}
plus a remainder term that is easy to handle.

For $\mu_l\not=0$ integration by parts in \eqref{4.14.1}  yields, as in Step 4, 
 terms of the form \eqref{expansion-terms} above,  but
here -- and only here -- we also encounter logarithmic terms, namely   
for the eigenvalue $\mu_1=0$, when 
\begin{eqnarray}
\label{homog_cdn}
%\lefteqn{
\fd+\fm k -2J+\gk_1=0. 
\end{eqnarray}
Note from \eqref{4.14.1} that the $\mu$-powers associated with log-terms are always integer multiples of $\fm$ as stated in \eqref{resolvent_exp}.  

Of particular interest is  the coefficient of $\mu^{-K\fm}\ln|\mu|$, since this yields the noncommutative residue. 
It is immediate from \eqref{4.14.1} that it can only arise for $\ell=k=0$. 
Together with Proposition \ref{4.2}, this gives us a wealth of information: 
By \ref{4.2}(b), the only terms with  $\ell=0$ in the parametrix construction yielding  the symbol of $q$ are of there form  $a_{\fa-j} (h_\fm+\mu^\fm)^{-K}$ for some term $a_{\fa-j}$ in the asymptotic expansion of the symbol of $A$. 
For $\tr_E q(\tilde Bw-\tilde b_0;\mu)$,  \ref{4.2}(d) shows that there may be contributions from the Taylor expansion, leading to terms $(-\tilde b_0)^\ga\partial^\ga_w\tr_E (a_{\fa-j}(\tilde Bw) (h_\fm(\tilde Bw)+\mu^\fm)^{-K})/\ga!$. 
However, none of the derivatives must fall on $(h_\fm(\tilde Bw)+\mu^\fm)^{-K}$, for then we would have $\ell\not=0$. 
Hence $d$ does not contain factors involving terms $h_{\fm-j}$ from the asymptotic expansion of $h$ or their derivatives. 
The fact that $k$ equals $0$ moreover implies that the integral \eqref{4.12.4}, and hence the coefficient of $\mu^{-K}\ln |\mu|$, is independent of the choice of the auxiliary operator $H$. 
This proves statement (e) in 
Theorem \ref{thm_resolvent}: The fact that 
$$\ln(-\gl) = \ln(\mu^\fm) = \fm \ln \mu $$
causes a factor $1/\fm=1/\ord H$ in the expansion with respect to $-\gl$.  

Finally, one has to treat the remainder terms in the expansion of $(h_\fm+\mu^\fm)^{-K-\ell}$. The analysis is slightly more involved, but similar to the above, and also produces terms of the form \eqref{expansion-terms}.
%}

\subsubsection*{Step 6. Correcting the expansion} 
We have  an expansion of $\Tr(R_gT_wA(H+\mu^\fm)^{-K})$ 
not only into powers of $\mu$ and $\ln \mu$, but also  in powers of $|\mu|$ times factors 
 of the form  $e^{i|\mu|^2\mu_l}g_l(\mu/|\mu|)$, $\mu_l\not=0$, or $h(\mu/|\mu|)$ that  are not of the form stated in Theorem \ref{thm_resolvent}. 
This is analogous to a phenomenon already observed by Grubb and Seeley. They excluded this possibility in  \cite[Lemma 2.3]{GS1}.
An argument along similar lines shows that the expansion only contains the terms listed in \eqref{resolvent_exp}. In particular, the terms involving factors $e^{i|\mu|^2\mu_l}$ for $\mu_l\not=0$  necessarily add up to zero.  

This last fact has also been established before in the context of the analysis of the zeta function  
$\Tr(R_gT_wAH^{-z})$ in \cite[Theorem 5]{SaSch3}.  

\subsubsection*{Step 7. Conclusion.} As the terms involving the eigenvalues $\mu_l\not=0$ all cancel and, moreover, only powers of $\mu$ and $\ln \mu$ occur, we see that we obtain the expansion \eqref{resolvent_exp}: 
In fact, we saw already how the terms $(c_j'\ln(-\gl) + c_j'')^{-K-j}$ arise.
As for the terms $c_j(-\gl)^{(2m+\fa-j)/\fm-K}$: they result from inserting the terms for $\mu_l=0$ in the stationary phase expansion (e.g. in \eqref{4.7.10} with $\hbar = r^{-2}|\mu|^{-2}$) into the corresponding integral \eqref{4.7.3}, noting that $\gk_1=2m$.
This completes  our sketch. \hfill $\Box$

%%%%%%%%%%%%%%%%%%%%%%%%%%%%%%%%%%%%%%%
%%%%%%%%%%%%%%%%%%%%%%%%%%%%%%%%%%%%%%%

\section{The Noncommutative Residue} \label{sect_residue} 

%%%%%%%%%%%%%%%%%%%%%%%%%%%%%%%%%%%%%%%
%%%%%%%%%%%%%%%%%%%%%%%%%%%%%%%%%%%%%%%

In Section \ref{sect_resolvent} we saw that  there is only one point, where logarithmic terms appear, namely in the analysis of the integral \eqref{4.12.4}. 
We also observed that, as a consequence of Proposition \ref{4.2}, the only possible coefficient functions of $(h_\fm(\tilde Bw)+\mu^\fm)^{-K}$ in the expansion of $\tilde q(\tilde Bw-\tilde b_0)$ are the functions 
\begin{eqnarray}\label{bq}
\frac{(-\tilde b_0)^\ga}{\ga!}  \partial_w^\ga \tr_E a_{\fa-j}(\tilde Bw)
\end{eqnarray}
for some choice of $j$ and $\ga$. In addition, the application of the stationary phase expansion may add further derivatives $A_{2J,1}(\gt',\partial_{\gt'})$ of order $2J$ on $a_{\fa-j}$ as we saw in \eqref{4.14.1} and \eqref{4.14.1a}.

Next,  Equation \eqref{homog_cdn} implies that a nonzero contribution to $\mu^{-K}\ln |\mu|$ will only arise if the homogeneity degree satisfies 
\begin{eqnarray}
\label{5.0.2}
%\lefteqn{
\fa-j-2J+\gk_1 =0,  
\end{eqnarray}
with $\gk_1=2m_5$.

In order to understand the contribution to the noncommutative residue, we distinguish the cases where $\fa= \ord A=-2m_5$, $\fa<-2m_5$ or $\fa>-2m_5$. 
 
(i) Let $\fa=-2m_5$. According to \eqref{5.0.2}, the only possibility  for obtaining a nontrivial contribution to the $\mu^{-K}\ln |\mu|$-term then is when $j=0$ and $\ga=0$ in \eqref{bq} and, moreover, $J=0$.
Setting $k=\ell=J=0$ and $l=1$ in \eqref{4.14.1} and \eqref{4.14.1a} shows that the coefficient is 
\begin{eqnarray}\label{cont0}
%\label{}
%\lefteqn{
\int_{\mathbb S^{2m_5-1}} \tr_E a_{-2m_5}(\tilde B\gt) dS. 
\end{eqnarray}
On the other hand, we know that $\tilde B$ is the identity on $(x_{n-m_5+1},p_{n-m_5+1},\ldots 
x_n,p_n)$ and that, on $\mathbb S^{2m_5-1}$ the other variables are zero. Hence 
\eqref{cont0} reduces to    
\begin{eqnarray}\label{cont1}
%\label{}
%\lefteqn{
\int_{\mathbb S^{2m_5-1}} \tr_E a_{-2m_5}(\gt)\, dS
\end{eqnarray}
as asserted. 

(ii) The same argument shows that the residue term vanishes if $\ord A<-2m_5$. 

(iii) If $\ord A>-2m_5$, then we reconsider Equation \eqref{4.12.4}.
The standard application of the stationary phase expansion  does not yield explicit expressions; therefore we take another approach. 
We know that there will be no contribution to the $\mu^{-K}\ln|\mu|$-term unless $k=\ell=0$ in  \eqref{4.12.4}.

For $\gk_1 = 2m_5$ we have 
$$\mathbb S^{2m_5-1} =\{ \gt \in \mathbb S^{n'-1} : 
\gt' = (\gt_1, \ldots , \gt_{n'-2m_5})=0\}.
$$

We introduce a parametrization of a tubular neighborhood of $\mathbb S^{2m_5-1}$ in 
$\mathbb S^{n'-1}$ as follows: We first choose a parametrization 
$\Phi: D\subset \R^{2m_5-1}\to \mathbb S^{2m_5-1}$ (e.g. the usual spherical coordinates) 
and then define, for small $\gve>0$ and $W_\gve = [-\gve,\gve]^{n'-2m_5}$, the parametrization
$$F: W_\gve\times D\to \R^{n'}, \quad F(\gt',\go) = (\gt', \Phi(\go)\sqrt{1-|\gt'|^2}).$$
The associated measure is $(1-|\gt'|^2)^{m_5-1} d\gt'dS$ with the surface measure $dS$ of
$\mathbb S^{2m_5-1}$. 

 Up to terms that furnish $O(\mu^{-\infty})$ contributions, the integral in \eqref{4.12.4} with $k=0$, $l=1$ and $d$ in \eqref{bq} of homogeneity degree $\fd = \fa-j-|\ga|$  is
\begin{eqnarray}
\label{oscint}
%\lefteqn{
\int_1^{|\mu|} \int_D\int_{W_\gve} e^{ir^2\skp{Q\gt',\gt'}} d(\gt',\Phi(\go)\sqrt{1-|\gt'|^2}) \, d\gt'dS\ r^{\fd+n'-1}\, dr,
\end{eqnarray}
where $Q={\diag} (\tilde \gl_1, \ldots, \tilde \gl_{n'-2m_5})$ with $\tilde \gl_j$ introduced after \eqref{eq:simple}.

The advantage of this representation is that we may apply to the inner integral the explicit version of the stationary phase expansion for quadratic exponentials with $\hbar=r^{-2}$, see  \cite[Theorem 3.13]{Z}: 
\begin{eqnarray*}
%\label{}
\lefteqn{\nonumber
\int e^{\frac ih \skp{Q\gt',\gt'}} d(\gt', \Phi(\go)\sqrt{1-|\gt'|^2}) \, d\gt'}\\
&=& \left(\frac{h}{2\pi}\right)^{\frac{n'-2m_5}2} \frac{e^{i\frac\pi4\sgn Q}}{|\det Q|^{1/2}}   
\left(
\sum_{J=0}^{N-1} \frac{h^J}{J!} \left( \frac{\skp{Q^{-1}D,D}}{2i}\right)^J d(0,\Phi(\go)) +O(h^N) \right).
\end{eqnarray*}
%Here,  $\det Q = \tilde \gl_1\cdots \tilde \gl_{n'-2m_5}$, 
Since $\hbar=r^{-2}$,  this furnishes an expansion into decreasing powers of $r$, which we then insert into   \eqref{oscint}. 
When evaluating the integral over $[1,|\mu|]$, we will only get a nontrivial contribution to the $\mu^{-K}\ln |\mu|$-term, if the total $r$-power is $-1$. 
In the case at hand, the possible powers are $(\fd+n'-1) +(-n'+2m_5)-2J = \fa-j-|\ga|+2m_5-2J-1$. Hence only those (finitely many) terms, where
$$\fa-j-|\ga| -2J = -2m_5$$
can contribute. For these we obtain an explicit formula. 

A further simplification comes from the fact that for $f$ positively homogeneous of degree $-2m_5$ on $\R^n\setminus \{0\}$  we have 
$$\int_{\mathbb S^{2m_5-1}} f(\gt) \, dS=0$$
whenever $f$ is a derivative of a positively homogeneous function of degree $-2m_5+1$, see  \cite[Lemma 1.2]{FGLS}. 
When considering the derivatives $\partial_w^\ga \tr_E a_{\fa-j}(\tilde Bw)$, we can therefore focus on those that are in directions normal to $\mathbb S^{2m_5-1}$.

\section{Appendix}

\subsection*{The parameter-dependent  calculus of Grubb and Seeley}

In order to make the text more accessible, we recall some results from \cite[Section 1]{GS1}. 

Let $S$ be an open sector in $\mathbb C$ with vertex at the origin and 
let $p=p(x,\xi,\mu)$ be a smooth function on $\mathbb R^\nu\times \mathbb R^n\times
\overline{S}$, which is additionally holomorphic for $\mu\in S$. 
%Here $S_\delta$ is the sector in \eqref{sector}. 

The symbol class $S^{m,0}(\mathbb R^\nu,\mathbb R^n;S)$ consists of all $p$    
for which 
\begin{eqnarray}\label{GS.1}%
\partial^j_tp(\cdot,\cdot, 1/t) \in S^{m+j} (\mathbb R^\nu\times \mathbb R^n) , \quad 1/t\in S, 
\end{eqnarray}
while  
$S^{m,d}(\mathbb R^\nu,  \mathbb R^n;S)$ consists of those $p$
for which 
$$\partial^j_t( t^dp(\cdot,\cdot, 1/t)) \in S^{m+j} (\mathbb R^\nu\times  \mathbb R^n), \quad 1/t\in S,$$
with uniform estimates  in  $S^{m+j}$ for $|t|\le 1$ and $\frac1t$ in closed 
subsectors of $S$.  

\begin{example}\label{ex.1}
If $a(x,\xi)$ is homogeneous of degree $m>0$ in $\xi$ for $|\xi|\ge1$ and smooth in $(x,\xi)$ on $\mathbb R^\nu\times \mathbb R^n$, and if $a(x,\xi)+\mu^m$ is invertible for $(x,\xi,\mu) \in \R^\nu\times\R^n\times \overline S$, then the symbol 
$p$ defined by 
$$p(x,\xi,\mu) = (a(x,\xi) +\mu^m)^{-1}$$
is an element of $ S^{0,-m}(\R^\nu, \R^n; S)\cap S^{-m,0}(\R^\nu, \R^n; S) $, see \cite[Theorem~1.17]{GS1}.  
\end{example}

\begin{theorem}\label{GrubbS2}
{\rm (a)} For $p_1\in S^{m_1,d_1}$ and $p_2\in S^{m_2,d_2}$ we have 
$p_1p_2\in S^{m_1+m_2,d_1+d_2}$. 
\\
{\rm (b)} For $p\in S^{m,d}$ set 
\begin{eqnarray}\label{exp_infty}% 
p_{(d,k)} (x,\xi) = \frac1{k!} \partial^k_t(t^dp(x,\xi,\frac1t))_{|t=0}.
\end{eqnarray}
Then $p_{(d,k)}\in S^{m+k} $, and for any $N$, 
\begin{eqnarray}%\label{}\lefteqn{}\\
p(x,\xi,\mu) -\sum_{k=0}^{N-1} \mu^{d-k} p_{(d,k)} (x,\xi) \in S^{m+N, d-N}.
\end{eqnarray}   
\end{theorem}

\begin{theorem}\label{parametrix}
Let $a\in S^m$, $m>0$, with principal symbol $a_m$ which is homogeneous of degree 
$m>0$ for $(x,\xi)$, $|\xi|\ge1$ and such that 
$a_m(x,\xi)+\mu^m$ is invertible for all $\mu$ in $S$.   
Then the standard parametrix construction furnishes a symbol $p=p(x,\xi,\mu)$ such that 
$$p\circ (a+\mu^m)  \sim 1 \sim (a+\mu^m)\circ p\text{\rm { mod }} 
S^{-\infty,-m}(\R^n, \R^n;S).$$
The symbol $p$ is an element of $S^{-m,0}\cap S^{0,-m}$ with an asymptotic 
expansion 
$$p\sim\sum_{j=0}^\infty p_{-m-j},$$
where
$%\begin{eqnarray*}
p_{-m}(x,\xi,\mu) = (a_m(x,\xi)+\mu^m)^{-1}  \in (S^{-m,0}\cap S^{0,-m})(\R^n, \R^n;S) $ 
and
$p_{-m-j} \in (S^{-m-j,0}\cap S^{m-j,-2m})(\R^n, \R^n;S)$.
\end{theorem}

In fact, the parametrix is constructed  by letting 
\begin{eqnarray*}%
p_{-m} (x,\xi,\mu) &=&(a_m(x,\xi)+\mu^m)^{-1}\quad\text{and} \\
p_{-m-j} (x,\xi,\mu)
&=&- \sum_{l,k,\ga}\frac1{\ga!}  (\partial_\xi^{\ga}p_{-m-l} ) 
(\partial_x^{\ga} a_{m-k} )p_{-m},
\end{eqnarray*}
where the sum is over all $k,l,\ga$ with $l<j$ and $l+k+|\ga|=j$. 
In fact, $p_{-m-j}$ is a linear combination of terms 
$p_{-m}(\partial_\xi^{\ga_1}\partial_x^{\gb_1} a_{m-k_1} ) 
%p_m(\partial_\xi^{\ga_2}\partial_x^{\gb_2}a_{m-k_2}) 
\ldots 
p_{-m}(\partial_\xi^{\ga_r}\partial_x^{\gb_r} a_{m-k_r} )p_{-m}$, 
where, for $j>0$, we have $r\ge 1$ and 
$%\begin{eqnarray*}%
\sum_{l=1}^r|\ga_l| =\sum_{l=1}^r |\gb_l| 
\text{ and } \sum_{l=1}^r|\ga_l|+\sum_{l=1}^r|k_l|=j .
$%\end{eqnarray*}

\bibliographystyle{amsplain}

\begin{thebibliography}{10}

\bibitem{BoNi} 
P.~Boggiatto and F.~Nicola. 
Non-commutative residues for anisotropic pseudo-differential operators in $\mathbb R^n$.
{\em J. Funct. Anal.}  203(2):305--320, 2003.

\bibitem{Bourbaki} 
N.~Bourbaki. 
\newblock Functions of a Real Variable. Elementary Theory. 
\newblock {\em Elements of Mathematics}. Springer-Verlag, Berlin, 2004.

\bibitem{Con1}
A.~Connes.
\newblock {\em Noncommutative geometry}.
\newblock Academic Press Inc., San Diego, CA, 1994.

\bibitem{CoMo1}
A.~Connes and H.~Moscovici.
\newblock The local index formula in noncommutative geometry.
\newblock {\em Geom. Funct. Anal.}, 5(2):174--243, 1995.

\bibitem{D1} S. Dave. An equivariant noncommutative residue. 
{\em J. Noncomm. Geom.} 7:709--735 (2013). 

\bibitem{deG1}
M.~de~Gosson.
\newblock {\em Symplectic geometry and quantum mechanics}, volume 166 of {\em
  Operator Theory: Advances and Applications}.
\newblock Birkh\"{a}user Verlag, Basel, 2006.

\bibitem{FGLS}
B.V.~Fedosov, F.~Golse, E.~Leichtnam, E.~Schrohe. 
\newblock The noncommutative residue for manifolds with boundary.
\newblock {\em J. Funct. Anal.}  142(1):1--31, 1996.

\bibitem{Gu}
V.~Guillemin. 
A new proof of Weyl's formula on the asymptotic distribution of eigenvalues.
\newblock {\em Adv. in Math.}  55(2):131--160, 1985.

\bibitem{GKN}
A. Gorokhovsky, N. de Kleijn and R. Nest. 
\newblock Equivariant algebraic index theorem. 
\newblock {\em J. Inst. Math. Jussieu}   20(3):929--955 (2021).

		
\bibitem{GSc1} 
G. Grubb and E. Schrohe. Trace expansions and the noncommutative residue for manifolds with boundary.
{\em  J. Reine Angew. Math.},  536:167--207, 2001.
		

\bibitem{GS1}
G.~Grubb and R.~T. Seeley.
\newblock Weakly parametric pseudodifferential operators and
  {A}tiyah-{P}atodi-{S}inger boundary problems.
\newblock {\em Invent. Math.}, 121(3):481--529, 1995.

\bibitem{GS2}
G.~Grubb and R.~T. Seeley.
\newblock Zeta and eta functions for {A}tiyah-{P}atodi-{S}inger operators.
\newblock {\em J. Geom. Anal.}, 6(1):31--77, 1996.

\bibitem{H95}
L.~H\"ormander. 
\newblock Symplectic classification of quadratic forms, and general Mehler formulas. 
\newblock {\em Math. Z.} 219:413--449, 1995. 

\bibitem{K89}
C. Kassel. Le r\'esidu non commutatif.
{\em Ast\'erisque}, tome 177-178 (1989), S\'eminaire Bourbaki, exp. no 708, p. 199--229.

\bibitem{Ler8}
J.~Leray.
\newblock {\em Analyse Lagrangienne et {M}{\'e}canique {Q}uantique}.
\newblock IRMA, Strasbourg, 1978.

\bibitem{Perrot14}
D.~Perrot. Local index theory for operators associated with {L}ie groupoid actions. 
J. Topology and Analysis, 2020. 
%arXiv:1401.0225v2 2016.
DOI: 10.1142/S1793525321500059.

\bibitem{SaSchSt4}
A.~Savin, E.~Schrohe, and B.~Sternin.
\newblock Elliptic operators associated with groups of quantized canonical
  transformations.
\newblock {\em {B}ull. {S}ci. {M}ath.}, 155:141--167, 2019.

\bibitem{SaSch1}
A.~Savin and E.~Schrohe.
\newblock Analytic and algebraic indices of elliptic operators associated with
  discrete groups of quantized canonical transformations.
\newblock {\em J. Funct. Anal.}, 278(5):108400, 45, 2020. 

\bibitem{SaSch2}
A.~Savin and E.~Schrohe. 
An index formula for groups of isometric linear canonical transformations.
arXiv:2008.00734. 

\bibitem{SaSch3}
A.~Savin and E.~Schrohe.
Local index formulae on noncommutative orbifolds and equivariant zeta functions for the affine metaplectic group. arXiv:2008.11075. 

\bibitem{Seeley67}
R.~T.~Seeley.  Complex powers of an elliptic operator.
In: Singular Integrals (Proc. Sympos. Pure Math., Chicago, Ill., 1966) 
pp. 288--307,  Amer. Math. Soc., Providence, R.I. 1967.

\bibitem{Shu1}
M.~A. Shubin.
\newblock {\em Pseudodifferential Operators and Spectral Theory}.
\newblock Springer--Verlag, Berlin--Heidelberg, 1985.

\bibitem{Wod2}
M.~Wodzicki.
\newblock Noncommutative residue. I. Fundamentals. 
\newblock {\em Lecture Notes in Math.}, 1289. Berlin, New York: Springer-Verlag, pp. 320--399, 1987.


\bibitem{Z}
M.~Zworski. {\em Semiclassical Analysis}. Graduate Studies in Mathematics {\bf 138}. Amer. Math. Soc., Providence, RI, 2012.
\end{thebibliography}

\end{document}